\documentclass[a4paper,12pt]{amsart}
\usepackage[english]{babel}
\usepackage{amssymb,amsfonts,amsxtra,amsmath}
\usepackage{dsfont,mathrsfs}
\renewcommand{\mathcal}{\mathscr}
\usepackage[all]{xypic}
\xyoption{dvips}
\usepackage{url}
\usepackage{fullpage}
\usepackage[colorlinks,plainpages,backref]{hyperref}
\usepackage[neveradjust]{paralist}

\theoremstyle{definition}
\newtheorem{ntn}{Notation}[section]

\theoremstyle{plain}
\newtheorem{lem}[ntn]{Lemma}
\newtheorem{prp}[ntn]{Proposition}
\newtheorem{thm}[ntn]{Theorem}
\newtheorem{cor}[ntn]{Corollary}

\newtheorem*{cnj*}{Conjecture}

\theoremstyle{remark}

\newtheorem{exa}[ntn]{Example}

\newtheorem{rmk}[ntn]{Remark}
\numberwithin{equation}{section}

\newcommand{\ideal}[1]{{\left\langle#1\right\rangle}}
\newcommand{\xymat}{\SelectTips{cm}{}\xymatrix}
\newcommand{\into}{\hookrightarrow}

\newcommand{\p}{\partial}

\newcommand{\wh}{\widehat}

\newcommand{\llangle}{\langle\langle}
\newcommand{\rrangle}{\rangle\rangle}

\renewcommand{\AA}{\mathbb{A}}
\newcommand{\fA}{\mathfrak{A}}
\renewcommand{\aa}{\mathfrak{a}}
\newcommand{\fB}{\mathfrak{B}}

\newcommand{\mm}{\mathfrak{m}}

\newcommand{\pp}{\mathfrak{p}}

\newcommand{\ZZ}{\mathds{Z}}

\DeclareMathOperator{\Der}{Der}
\DeclareMathOperator{\embdim}{embdim}

\DeclareMathOperator{\Hom}{Hom}

\DeclareMathOperator{\rk}{rk}

\DeclareMathOperator{\Sing}{Sing}
\DeclareMathOperator{\Spec}{Spec}

\begin{document}
%%%%%%%%%%%%%%%%%%%%%%%%%%%%%%%%%%%%%%%%%%%%%%%%%%%%%%%%%%%%%%%%%%%%%%%%%%%%%%%

\title[Derivations on complete intersections]{Derivations of negative degree on quasihomogeneous isolated complete intersection singularities}

\author[M.~Granger]{Michel Granger}
\address{
M.~Granger\\
Universit\'e d'Angers, D\'epartement de Math\'ematiques\\ 
LAREMA, CNRS UMR n\textsuperscript{o}6093\\ 
2 Bd Lavoisier\\ 
49045 Angers\\ 
France}
\email{\href{mailto:granger@univ-angers.fr}{granger@univ-angers.fr}}
%\thanks{}

\author[M.~Schulze]{Mathias Schulze}
\address{M.~Schulze\\
Department of Mathematics\\
University of Kaiserslautern\\
67663 Kaiserslautern\\
Germany}
\email{\href{mailto:mschulze@mathematik.uni-kl.de}{mschulze@mathematik.uni-kl.de}}
\thanks{The research leading to these results has received funding from the People Programme (Marie Curie Actions) of the European Union's Seventh Framework Programme (FP7/2007-2013) under REA grant agreement n\textsuperscript{o} PCIG12-GA-2012-334355.}

\date{\today}

% 13N15 Derivations
% 14M10 Complete intersections
% 14H20 Singularities, local rings
% 32S25 Surface and hypersurface singularities
% 16W25 Derivations, actions of Lie algebras
% 17d66 Lie algebras of vector fields and related (super) algebras


\begin{abstract}
J.~Wahl conjectured that every quasihomogeneous isolated normal singularity admits a positive grading for which there are no derivations of negative weighted degree. 
We confirm his conjecture for quasihomogeneous isolated complete intersection singularities of either order at least $3$ or embedding dimension at most $5$.
For each embedding dimension larger than $5$ (and each dimension larger than $3$), we give a counter-example to Wahl's conjecture.
\end{abstract}

\maketitle
\tableofcontents

%%%%%%%%%%%%%%%%%%%%%%%%%%%%%%%%%%%%%%%%%%%%%%%%%%%%%%%%%%%%%%%%%%%%%%%%%%%%%%%
\section{Introduction}
%%%%%%%%%%%%%%%%%%%%%%%%%%%%%%%%%%%%%%%%%%%%%%%%%%%%%%%%%%%%%%%%%%%%%%%%%%%%%%%

By a singularity we mean a quotient $A$ of a convergent power series ring over a valued field $K$ of characteristic zero (see \S\ref{27}).
We use the acronym \emph{negative derivation} for a derivation of negative weighted degree on a quasihomogeneous singularity.
The question of existence of such negative derivations has important consequences in rational homotopy theory (see \cite[Thm. A]{Mei81}) and in deformation theory (see~\cite[Thm.~3.8]{Wah81}).

By a result of Kantor~\cite{Kan79}, quasihomogeneous curve and hypersurface singularities do not admit any negative derivations. 
J.~Wahl~\cite[Thm.~2.4, Prop.~2.8]{Wah81} reached the same conclusion in (the much deeper) case of quasihomogeneous normal surface singularities.
Motivated by his cohomological characterization of projective space in \cite{Wah83a}, he formulates the following conjecture in \cite[Conj.~1.4]{Wah83b}.

\begin{cnj*}[Wahl]
Let $R$ be a normal graded ring, with isolated singularity.
Then there is a normal graded $\bar R$, with $\hat R\cong\hat{\bar R}$, so that $\bar R$ has no derivations of negative weight.
\end{cnj*}

In case $R$ is a graded normal locally complete intersection with isolated singularity, $\hat R$ becomes a quasihomogeneous normal isolated complete intersection singularity (ICIS) and Wahl's conjecture can be rephrased as follows (see Lemma~\ref{46} and Remark~\ref{48}).

\begin{cnj*}[Wahl, ICIS case]
Any quasihomogeneous normal ICIS has no negative derivations with respect to some positive grading.
\end{cnj*}

For quasihomogeneous normal ICIS, there is an explicit description of all derivations due to Kersken~\cite{Ker84}.
Based on this description, we prove our main

\begin{thm}\label{0}
For any quasihomogeneous normal ICIS of order at least $3$ there are no negative derivations with respect to any positive grading.
\end{thm}

\begin{proof}
This follows from Corollary~\ref{41} and Proposition \ref{14}.
\end{proof}

Our investigations lead to a family of counter-examples to Wahl's Conjecture.
In order to describe it, we first fix our notation.
A quasihomogeneous singularity can be represented as
\begin{equation}\label{24}
A=P/\aa,\quad \aa=\ideal{g_1,\dots,g_t}\unlhd K\llangle x_1,\dots,x_n\rrangle=:P
\end{equation}
where $g_1,\dots,g_t$ are homogeneous polynomials of degree $p_i:=\deg(g_i)$ with respect to weights $w_1,\dots,w_n\in\ZZ_+$ on the variables $x_1,\dots,x_n$ (see \S\ref{27}).
We order these weights and degrees decreasingly as
\begin{gather}\label{1}
w_1\ge\dots\ge w_n>0,\\
\nonumber p_1\ge\dots\ge p_t.
\end{gather}

\begin{exa}\label{22}
Let $n\ge6$ and pick $c_7,\dots,c_n\in K\setminus\{1\}$ pairwise different such that $c_i^9+1\neq0$ for all $i$.
Assigning weights $8,8,5,2,\dots,2$ to the variables $x_1,\dots,x_n$, the equations
\begin{align}\label{20}
g_1:=&x_1x_4+x_2x_5+x_3^2-x_4^5+\sum_{i=7}^nx_i^5\\
\nonumber g_2:=&x_1x_5+x_2x_6+x_3^2+x_6^5+\sum_{i=7}^nc_ix_i^5
\end{align}
define a quasihomogeneous complete intersection $A$ as in \eqref{24} with isolated singularity.
On $A$ there is a derivation
\begin{equation}\label{21}
\eta:=
\begin{vmatrix}
\p_1 & \p_2 & \p_3 \\
x_4 & x_5 & 2x_3 \\
x_5 & x_6 & 2x_3 
\end{vmatrix}
=2x_3(x_5-x_6)\p_1-2x_3(x_4-x_5)\p_2+(x_4x_6-x_5^2)\p_3
\end{equation}
of degree $-1$.
We work out the details of this example in \S\ref{18}.
\end{exa}

We show that Example~\ref{22} gives a counter-example to the ICIS case of Wahl's conjecture of minimal embedding dimension $n=6$.

\begin{thm}
Exactly up to embedding dimension $5$, all quasihomogeneous ICIS have no negative derivations with respect to some positive grading.
\end{thm}

\begin{proof}
This follows from Kantor~\cite{Kan79}, \cite[Thm.~2.4, Prop.~2.8]{Wah81}, Proposition~\ref{38}, Example~\ref{22} and Corollary~\ref{41}.
\end{proof}

As a consequence of our arguments we obtain a simple special case of the following conjecture due to S.~Halperin.

\begin{cnj*}[Halperin]
On any graded zero-dimensional complete intersection there are no negative derivations.
\end{cnj*}

The following result bounds the degree of negative derivations (see also \cite[Prop.]{Ale91}).
The bound does not require a complete intersection hypothesis and it is independent of further hypotheses as for instance in \cite[Thm.~2]{Hau02}.

\begin{prp}
For any quasihomogeneous zero-dimensional singularity $A$ as in \eqref{24} there are no derivations of degree strictly less than $p_n-p_1$.
In particular, Halperin's conjecture holds true if $p_1=p_n$.
\end{prp}

\begin{proof}
As $A$ is assumed to be zero-dimensional, condition $\fA(k)$ on page \pageref{32} must hold true for all $k=1,\dots,n$.
Then the claim follows from Remark~\ref{30} and Lemma~\ref{13}.
\end{proof}

%%%%%%%%%%%%%%%%%%%%%%%%%%%%%%%%%%%%%%%%%%%%%%%%%%%%%%%%%%%%%%%%%%%%%%%%%%%%%%%
\subsubsection*{Acknowledgments}
%%%%%%%%%%%%%%%%%%%%%%%%%%%%%%%%%%%%%%%%%%%%%%%%%%%%%%%%%%%%%%%%%%%%%%%%%%%%%%%

The second author would like to thank the LAREMA at the University of Angers 
for providing financial support and a pleasant working atmosphere during his research visit in February 2014.

%%%%%%%%%%%%%%%%%%%%%%%%%%%%%%%%%%%%%%%%%%%%%%%%%%%%%%%%%%%%%%%%%%%%%%%%%%%%%%%
\section{Graded analytic algebras}\label{27}
%%%%%%%%%%%%%%%%%%%%%%%%%%%%%%%%%%%%%%%%%%%%%%%%%%%%%%%%%%%%%%%%%%%%%%%%%%%%%%%

Consider a (local) analytic algebra $A=(A,\mm_A)$ over a (possibly trivially) valued field $K$ of characteristic zero.
We assume in addition that $A$ is non-regular and can be represented as a quotient $A=P/\aa$ of a convergent power series ring $P:=K\llangle x_1,\dots,x_n\rrangle\unrhd\aa$.
In the sequel such an $A$ will be referred to as a \emph{singularity}.
We choose $n$ minimal such that $n=\embdim A$ and set $d:=\dim A$.
 
A $K_+$-grading on $A$ is given by a \emph{diagonalizable derivation} $\chi\in\Der_KA=:\Theta_A$ which means that $\mm_A$ is generated by eigenvectors $x_1,\dots,x_n$ (see \cite[(2.2),(2.3)]{SW73}).
Such a derivation is also called an \emph{Euler derivation}.
We refer to $w_1,\dots,w_n$ defined by $w_i:=\chi(x_i)/x_i$ as the \emph{eigenvalues of $\chi$}.
More generally, we call $\chi$-eigenvectors $f\in A$ \mbox{\emph{($\chi$-)homogeneous}} and define their \mbox{\emph{($\chi$-)degree}} to be the corresponding eigenvalue denoted by $\deg(f):=\chi(f)/f\in k$.
We denote by $A_a$ the $K$-vector space of all such eigenvector $f\in A$ with $\deg(f)=a$.
This defines a $K$-subalgebra
\begin{equation}\label{50}
\bar A:=\bigoplus_{a\in K}A_a\subset A\subset\hat A.
\end{equation}

The derivation $\chi\in\Theta_A$ lifts to $\chi\in\Theta_P:=\Der_KP$ (see \cite[(2.1)]{SW73}).
In particular, $P$ is $K_+$-graded and $\aa\unlhd P$ is a $\chi$-invariant ideal and hence ($\chi$-)homogeneous (see \cite[(2.4)]{SW73}).
Pick homogeneous $g_1,\dots,g_t\in\aa$ inducing a $K$-vector space basis of $\aa/\mm_A\aa$.
Then $\aa=\ideal{g_1,\dots,g_t}$ by Nakayama's Lemma.
We set $p_i:=\deg(g_i)$ ordered as in \eqref{1}.
To summarize, we can write $A$ as in \eqref{24}.

A $K_+$-grading is called a positive grading if $w_i\in\ZZ_+$ for all $i=1,\dots,n$ (see \cite[\S3, Def.]{SW73}). 
We call $A$ \emph{quasihomogeneous} if it admits a positive grading.
In this case, we shall always normalize $\chi$ to make the $w_i$ coprime and order the variables according to \eqref{1}.
Positivity of weights enforces $g_i\in\bar P=K[x_1,\dots,x_n]$ and that
\begin{equation}\label{36}
\bar A=\bigoplus_{i\ge0}A_i=\bar P/\bar\aa,\quad\bar\aa=\ideal{g_1,\dots,g_t}\unlhd K[x_1,\dots,x_n]=\bar P,
\end{equation}
is a (positively) graded-local $k$-algebra with completion
\begin{equation}\label{47}
\hat{\bar A}=\hat A
\end{equation}
and graded maximal ideal $\mm_{\bar A}=\bar\mm_A:=\bigoplus_{i>0}A_i$.
The preceding discussion enables us to reformulate Wahl's Conjecture in the language of Scheja and Wiebe as follows.

\begin{lem}\label{46}
The following supplementary structures on a singularity $A$ are equivalent:
\begin{enumerate}
\item\label{46a} an Euler derivation $\chi$ on $A$ with positive eigenvalues,
\item\label{46b} a positive grading on $A$,
\item\label{46c} a positive grading on $\hat A$,
\item\label{46d} a (positively) graded $K$-algebra $\bar A$ such that $\hat{\bar A}=\hat A$.
\end{enumerate}
\end{lem}

\begin{proof}
The equivalences of \eqref{46a}, \eqref{46b}, and \eqref{46c} are due to Scheja and Wiebe (see \cite[(2.2),(2.3)]{SW73} and \cite[(1.6)]{SW77}).
For the equivalence with \eqref{46d}, note that the obvious Euler derivation on a graded $K$-algebra $\bar A$ lifts to an Euler derivation on the completion $\hat{\bar A}=\hat A$.
The converse follows from from \eqref{50}, \eqref{36} and \eqref{47}.
\end{proof}

Let us assume now that $A$ is an isolated complete intersection singularity (ICIS).
We may then take $g_1,\dots,g_t$ to be a regular sequence and $d+t=n$.
The isolated singularity hypothesis can be expressed in terms of the Jacobian ideal
\begin{equation}\label{25}
J_A:=\ideal{\left\vert\frac{\p g}{\p x_\nu}\right\vert\mid |\nu|=t}\unlhd A
\end{equation}
of $A$ as follows.

\begin{prp}\label{37}
A complete intersection singularity $A$ is isolated if and only if $J_A$ is $\mm_A$-primary.
An analogous statement holds for $\bar A$.
\end{prp}

\begin{proof}
We denote by $\Omega^1_{A/k}$ the universally finite module of differentials of $A$ over $k$.
By the standard sequence
\[
\xymat{
\aa/\aa^2\ar[r] & A\otimes_P\Omega^1_{P/k}\ar[r] & \Omega^1_{A/k}\ar[r] & 0,
}
\]
the Jacobian ideal $J_A$ is the $0$th Fitting ideal $F^0_A\Omega^1_{A/k}$.
By \cite[(6.4),(6.9)]{SS72}, reducedness of $A$ is equivalent to $\rk\Omega^1_{A/k}=d$ and $A_\pp$ is regular if and only if $\Omega^1_{A_\pp/k}$ is free.
Hence, $A_\pp$ being regular is equivalent to $\pp\not\supset F^0_A\Omega^1_{A/k}=J_A$ by \cite[Lem.~1.4.9]{BH93}.
In particular, $A$ having an isolated singularity means exactly that $A/J_A$ is supported at $\mm_A$ and hence that $J_A$ is $\mm_A$-primary as claimed.
The analogous statement for $\bar A$ is proved similarly.
\end{proof}

\begin{rmk}\label{48}
Let $A$ be a quasihomogeneous singularity.
By \eqref{36},
\begin{equation}\label{49}
J_{\bar A}:=\bar J_A=\ideal{\left\vert\frac{\p g}{\p x_\nu}\right\vert\mid |\nu|=t}\unlhd\bar A
\end{equation}
is the Jacobian ideal of $\bar A$ defined analogous to \eqref{25}.
By \eqref{47}, $A$ is a complete intersection if and only if $\bar A$ is locally a complete intersection (see \cite[Def.~2.3.1, Ex.~2.3.21.(c)]{BH93}).
By Proposition~\ref{37}, $A$ is an ICIS if and only if $J_A$ is $\mm_A$-primary.
This is equivalent to $J_{\bar A}$ being $\mm_{\bar A}$-primary.
The latter is then equivalent to $\bar A$ being locally a complete intersection with isolated singularity by \eqref{49} and Proposition~\ref{37}.
Complete intersections are Cohen--Macaulay and hence $(S_2)$ so normality is equivalent to $(R_1)$ by Serre's Criterion (see \cite[\S2.3, Thm.~2.2.22]{BH93}). 
Since $d=\dim A=\dim\bar A$ by \eqref{47} (see \cite[Cor.~2.1.8]{BH93}), normality for both $A$ and $\bar A$ reduces to $d\ge2$.
\end{rmk}

Scheja and Wiebe~\cite[(3.1)]{SW77} (see also \cite[Satz~1.3]{Sai71}) proved that any $K_+$-graded ICIS is quasihomogeneous unless $t=1$ and $g_1\notin\mm_P^3$.
Their starting point (see \cite[(2.5)]{SW77} and \cite[Lem.~1.5]{Sai71}) is that $A$ being an ICIS implies, by Proposition~\ref{37}, that for each $k=1,\dots,n$ one of the following two conditions must holds true.
\begin{enumerate}
\item[$\fA(k)$]\label{32} For some $m\ge2$ and $1\le j\le t$, the monomial $x_k^m$ occurs in $g_j$.
\item[$\fB(k)$] For some pairwise different $1\le\nu_1,\dots,\nu_t\le n$, each $g_j$ contains a monomial $x_k^{m_j}x_{\nu_j}$ for some $m_j\ge1$.
\end{enumerate}
The following result gives numerical constraints for $A$ to be a quasihomogeneous ICIS.

\begin{lem}\label{12}
If $A$ is a quasihomogeneous ICIS then
\begin{equation}\label{10}
p_1+\cdots+p_j\ge w_1+\dots+w_j+j
\end{equation}
for all $j=1,\dots,t$.
\end{lem}

\begin{proof}
We proceed by induction on $j$.
Assume that $p_1+\cdots+p_{j-1}\ge w_1+\dots+w_{j-1}+j-1$ but $p_1+\cdots+p_j\le w_1+\dots+w_j+j-1$.
Then $p_j\le w_j$ and hence $g_i=g_i(x_{j+1},\dots,x_n)$ for all $i=j,\dots,n$.
Then $J_A$ maps to zero in
\[
A/\ideal{x_{j+1},\dots,x_n}=K\llangle x_1,\dots,x_j\rrangle/\ideal{g_1,\dots,g_{j-1}}
\]
and hence $J_A$ cannot be $\mm_A$-primary as required by Proposition~\ref{37}.
\end{proof}

%%%%%%%%%%%%%%%%%%%%%%%%%%%%%%%%%%%%%%%%%%%%%%%%%%%%%%%%%%%%%%%%%%%%%%%%%%%%%%%
\section{Negative derivations}
%%%%%%%%%%%%%%%%%%%%%%%%%%%%%%%%%%%%%%%%%%%%%%%%%%%%%%%%%%%%%%%%%%%%%%%%%%%%%%%

Let $A$ be a quasihomogeneous singularity as in \S\ref{27}.
The target of our investigations is the positively graded $A$-module $\Theta_A=\Der_KA$ of $K$-linear derivations on $A$.
More precisely, we are concerned with the question whether its negative part
\[
\Theta_{A,<0}=\Theta_{\bar A,<0}=\bigoplus_{i<0}\Theta_{A,i}
\]
is trivial.
A priori this condition depends on the choice of a grading.
In Proposition~\ref{19} below, we shall prove the independence of this choice for a general singularity under a strong hypothesis satisfied in the ICIS case (see Corollary~\ref{41}). 
To this end, we write (see \cite[(2.1)]{SW73})
\begin{equation}\label{42}
\Theta_A=\Theta_{\aa\subset P}/\aa\Theta_P
\end{equation}
as a quotient of a $(k,P)$-Lie algebra
\[
\Theta_{\aa\subset P}:=\{\delta\in\Theta_P\mid\delta\aa\subset\aa\}\unrhd\aa\Theta_P
\]
of \emph{logarithmic derivations} along $\aa$ by the $(k,P)$-Lie ideal $\aa\Theta_P$.

\begin{prp}\label{19}
Let $A$ be a quasihomogeneous singularity with positive grading given by $\chi$ and assume that
\begin{align}
\Theta_{\aa\subset P}&=P\chi+\Theta'_P+\aa\Theta_P\label{43},\\
\text{for some}\quad\Theta'_P&\subset\mm_P^2\Theta_P\label{44}.
\end{align}
Then the condition $\Theta_{A,<0}=0$ and the $p_1,\dots,p_t$ in \eqref{1} are independent of the chosen positive grading.
\end{prp}

\begin{proof}
Consider a second positive grading with corresponding Euler derivation $\chi'$ (see Lemma~\ref{46}).
By \eqref{42} and \eqref{43}, any $\delta\in\Theta_A$ lifts to an element of $\Theta_{\aa\subset P}$ of the form 
\begin{equation}\label{39}
\delta=c\chi+\delta_+,\quad\delta_+=a\chi+\eta,\quad c\in K,\quad a\in\mm_P,\quad \eta\in\Theta'_P,
\end{equation}
denoted by the same symbol.
By \eqref{44} and the Leibniz rule, 
\begin{equation}\label{40}
\chi\mm_P^k\subset\mm_P^k,\quad\delta_+\mm_P^k\subset\mm_P^{k+1}
\end{equation}
for all $k\ge1$.
Specializing to $\delta=\chi$, this implies that $\chi_+=0$ and $\chi'=c\chi$ on $\mm_A/\mm_A^2=\mm_P/\mm_P^2$ and hence $c=1$ by the definition of a positive grading and our normalization of weights.

Using \eqref{42}, we equip $\Theta_A$ with the decreasing $\mm_P$-adic filtration $F^\bullet$ induced from $\Theta_P$ which is defined as follows
\[
F^k\Theta_A=(\Theta_{\aa\subset P}\cap\mm_P^k\Theta_P)/(\aa\Theta_P\cap\mm_P^k\Theta_P).
\]
Due to \eqref{44}, \eqref{39} and \eqref{40} this is a filtration by $(k,P)$-Lie ideals and
\[
\delta_+F^k\Theta_A\subset F^{k+1}\Theta_A
\]
for the adjoint action of $\delta_+$.
Therefore, for any $k\ge1$, the adjoint action of $\chi'=\chi+\chi_+$ on the truncation
\[
F^{\le k}\Theta_A:=\Theta_A/F^{k+1}\Theta_A
\]
is triangularizable with semisimple part equal to that of $\chi$.
Thus, $\chi'$ and $\chi$ have the same eigenvalues on $F^{\le k}\Theta_A$ for any $k\ge1$.
The first claim then follows by choosing $k$ sufficiently large.
A similar argument yields the second claim.
\end{proof}

For a Gorenstein singularity $A$, there is a natural way to produce elements of $\Theta_A$.
The $A$-submodule $\Theta_A'\subset\Theta_A$ of \emph{trivial derivations} is by definition the image of the inclusion
\begin{equation}\label{51}
\Omega_{A/K}^{d-1}\into\omega_{A/K}^{d-1}=\Hom_A(\Omega_{A/K}^1,\omega_{A/K}^d)=\Theta_A\otimes_A\omega_{A/K}^d\cong\Theta_A.
\end{equation}
We return to the case of an ICIS singularity $A$.
For $1\le\nu_0<\dots<\nu_t\le n$ with complementary indices $1\le\mu_1<\dots<\mu_{d-1}\le n$, the lift to $P$ of the image of $dx_{\mu_1}\wedge\dots\wedge dx_{\mu_{d-1}}$ can be written (up to sign) explicitly as 
\begin{equation}\label{3}
\delta_\nu:=
\begin{vmatrix}
\p_{\nu_0} & \cdots & \p_{\nu_t}\\
\p_{\nu_0}g_1 & \cdots & \p_{\nu_t}g_1\\
\vdots & & \vdots \\
\p_{\nu_0}g_t & \cdots & \p_{\nu_t}g_t
\end{vmatrix}.
\end{equation}
Note that 
\begin{align}
\deg\delta_\nu&=p_1+\dots+p_t-w_{\nu_0}-\dots-w_{\nu_t}\label{2},\\
\delta_\nu g_j&=0\label{4}
\end{align}
for all $j=1,\dots,t$ and $\nu$.
Consider the $P$-module
\begin{equation}\label{45}
\Theta'_P:=\ideal{\delta_\nu\mid1\le\nu_0<\dots<\nu_t\le n}_P\subset\Theta_P.
\end{equation}
The key to our investigations is the following result due to Kersken~\cite[(5.2)]{Ker84}.
From now on we assume in addition that $A$ is quasihomogeneous and normal, that is, $\dim A\ge2$.

\begin{thm}[Kersken]\label{17}
Let $A$ be a quasihomogeneous normal ICIS.
Then the module $\Theta_A$ of $K$-linear derivations on $A$ is generated by the Euler derivation $\chi$ and the trivial derivations $\Theta_A'$.
\end{thm}

Although Kersken only states that $\Theta'_A$ is minimally generated by the $\delta_\nu$ in \eqref{3}, his arguments show that together with $\chi$ they form a minimal set of generators of $\Theta_A$.
We denote by $\mu(-)$ the minimal number of generators.

\begin{cor}
Let $A$ be quasihomogeneous normal ICIS. 
Then $\Theta_A$ is minimally generated by the Euler derivation $\chi$ and the trivial derivations $\delta_\nu$ in \eqref{3}.
In particular,
\[
\mu(\Theta_A)={n\choose t+1}+1.
\]
\end{cor}

\begin{proof}
Since the case $d=2$ is covered by \cite[Prop.~1.12]{Wah87}, we may assume that $d\ge3$.
In this case, the inclusion \eqref{51} fits into the following commutative diagram with exact rows and columns (see \cite[Proof of (4.8)]{Ker84} or \cite[Prop.~1.7]{Wah87}).
\begin{equation}\label{52}
\xymat{
& 0 & 0\\
0\ar[r] & H_{\mm_A}^1(\Omega^d_{A/K})\ar[u]\ar[r]^-\chi_-\cong & H_{\mm_A}^1(\Omega^{d-1}_{A/K})\ar[u]\ar[r] & 0 \\
0\ar[r] & \omega_{A/K}^d\ar[u]\ar[r]^-\chi & \omega_{A/K}^{d-1}\ar[u]\ar[r]^-\chi & \omega_{A/K}^{d-2}\ar[u]\\
0\ar[r] & \Omega_{A/K}^d\ar[u]\ar[r]^-\chi & \Omega_{A/K}^{d-1}\ar[u]\ar[r]^-\chi & \Omega_{A/K}^{d-2}\ar[u]_-\cong\\
&&0\ar[u]&0\ar[u]
}
\end{equation}
It follows that 
\[
\chi(\omega_{A/K}^{d-1})\cong\chi(\Omega_{A/K}^{d-1})\cong\Omega_{A/K}^{d-1}/\chi(\Omega_{A/K}^d)
\]
where $\chi(\Omega_{A/K}^d)\subset\mm_A\Omega_{A/K}^{d-1}$ and hence
\[
\mu(\chi(\Omega_{A/K}^{d-1}))=\mu(\Omega_{A/K}^{d-1})=\mu(\Theta'_A).
\]
Now the middle row of \eqref{52} yields an exact sequence
\[
\xymat{
0\ar[r] & A\ar[r]^-\chi & \Theta_A\ar[r] & \chi(\omega_{A/K}^{d-1})\otimes(\omega_{A/K}^d)^{-1}\ar[r] & 0
}
\]
Since $\chi\notin\mm_A\Theta_A$, the claim follows.
\end{proof}

Note that $\Theta'_P$ in \eqref{45} satisfies \eqref{44} due to \eqref{3} unless $t=1$ and $g_1\notin\mm_P^3$.
As a consequence of Proposition~\ref{19} and Theorem~\ref{17} we therefore obtain the following result.
It is crucial for Example~\ref{22} to be a counter-example to Wahl's Conjecture.

\begin{cor}\label{41}
Let $A$ be a quasihomogeneous normal ICIS.
Unless $t=1$ and $g_1\notin\mm_P^3$, the condition $\Theta_{A,<0}=0$ and the $p_1,\dots,p_t$ in \eqref{1} are independent of the choice of a positive grading.\qed
\end{cor}

We shall now derive numerical constraints for minimal negative trivial derivations.
To this end, suppose that $0\ne\eta\in\Theta_{A,<0}$.
For degree reasons (see \eqref{1}), $\eta$ can be written as 
\begin{equation}\label{7}
\eta=q_1\p_1+\cdots +q_n\p_n,\quad q_i=q_i(x_{i+1},\dots,x_n)
\end{equation}
By Theorem~\ref{17}, we may assume that $\eta=\delta_\nu\ne0$ is a trivial derivation as in \eqref{3}.
By \eqref{1} and \eqref{2}, we may further assume that $\nu_i=i+1$ for all $i=0,\dots,t$ and hence $q_i=0$ for all $i>t+1$.
The preceding arguments combined with \eqref{2} and \eqref{4}, can be summarized as follows.

\begin{lem}\label{54}
Let $A$ be a quasihomogeneous normal ICIS.
Then, for all $\eta\in\Theta_{A,<0}$ and all $j=1,\dots,t$, we have
\begin{equation}\label{6}
\eta g_j=0.
\end{equation}
If $\Theta_{A,<0}\ne0$ then there is a derivation $0\ne\eta\in\Theta_{A,<0}$ as in \eqref{7} with $q_i=0$ for all $i>t$.
It gives rise to an inequality
\begin{equation}\label{5}\pushQED{\qed}
p_1+\dots+p_t<w_1+\dots+w_{t+1}.\qedhere
\end{equation}
\end{lem}

\begin{rmk}\label{30}
For degree reasons (see \eqref{1}), the identity \eqref{6} holds true for any $\eta\in\Theta_{A,<p_t-p_1}$ and any quasihomogeneous singularity $A$ as in \eqref{24}.
\end{rmk}

We now link the conditions $\fA(k)$ and $\fB(k)$ from page \pageref{32} to the existence of a negative derivation.

\begin{lem}\label{13}
Assume that $\eta\in\Theta_{A,<0}$ as in \eqref{7} with $q_i=0$ for all $i\in I\supset\{0,\dots,k-1\}$ satisfies \eqref{6} for all $j=1,\dots,t$ and that $\fA(k)$ holds true.
Then there is a $\chi$-homogeneous coordinate change preserving the preceding conditions which makes $q_k=0$.
\end{lem}

\begin{proof}
As $\fA(k)$ holds by hypothesis, there is a $g:=g_j$ containing a monomial $x_k^m$, $m>1$.
Expanding respect to $x_k$,
\[
g=\sum_{j=0}^m t_jx_k^{m-j},\quad t_j=t_j(x_1,\dots,\wh{x_k},\dots,x_n).
\]
We may normalize $g$ such that $t_0=\frac1m$.
Note that $t_j$ is homogeneous of degree $j\cdot w_k$ and, in particular, $\deg(t_1)=\deg(x_k)$.
Then, ordering terms according to $i=k$ or $i>k$, \eqref{6} becomes
\begin{align}\label{53}
0=\eta(g)&=\sum_{i\ge k}\sum_{j=0}^m q_i\p_i(t_jx_k^{m-j})\\
\nonumber&=\sum_{j=1}^m\left((m-j+1)q_kt_{j-1}+\sum_{i>k}q_i\p_i(t_j)\right)x_k^{m-j}.
\end{align}
By \eqref{7}, all $q_i$, $i\ge k$, are independent of $x_k$.
Thus, using $t_0=\frac1m$, the coefficient equation of $x_k^{m-1}$ in \eqref{53} reads
\[
q_k+\sum_{i>k}q_i\p_i(t_1)=0
\]
and $\eta$ can be rewritten as
\begin{equation}\label{9}
\eta=\sum_{i>k}q_i\cdot(\p_i-\p_i(t_1)\p_k).
\end{equation}
The $\chi$-homogeneous coordinate change
\[
x_i'=
\begin{cases}
x_k+t_1,&\text{ if }i=k,\\
x_i,&\text{ otherwise,}
\end{cases}
\]
replaces $\p_i-\p_i(t_1)\p_k$ in \eqref{9} by $\p_i'$, and thus $q_k$ in \eqref{7} by $0$.
\end{proof}

Our main technical result is the following

\begin{prp}\label{14}
Let $A$ be a quasihomogeneous normal ICIS such that $\Theta_{A,<0}\ne0$.
Then $\fB(k)$ holds true for some $k\le t$ after some $\chi$-homogeneous coordinate change.
Each such $k$ satisfies $k\ge t-d+2$ and $g_k,\dots,g_t\notin\mm_P^3$.
\end{prp}

\begin{proof}
Let $0\ne\eta\in\Theta_{A,<0}$ be given by Lemma~\ref{54}.
For increasing $k\ge1$ with $q_k\ne0$, we repeatedly apply Lemma~\ref{13} with $I=\{1,\dots,k-1,t+2,\dots,n\}$ as long as $\fA(k)$ holds.
The procedure stops with 
\begin{equation}\label{55}
0\ne\eta=q_k\p_k+\dots+q_{t+1}\p_{t+1},\quad q_k\ne0,
\end{equation}
for some $k\le t+1$ by choice of $\eta$ and Lemma~\ref{13}. 
Since $\fA(k)$ fails, $\fB(k)$ must hold true.
In case $k=t+1$ in \eqref{55}, \eqref{6} becomes $\p_kg_j=0$ for all $j=1,\dots,t$.
This would mean that all $g_j$ are independent of $x_k$ in contradiction to the isolated singularity hypothesis.
Therefore, $k\le t$ as claimed.

Combining \eqref{10} and \eqref{5}, we obtain
\begin{equation}\label{11}
p_j+\dots+p_t+j\le w_j+\dots+w_{t+1}
\end{equation}
for all $j=1,\dots,t$.
Using \eqref{1}, $\fB(k)$ and \eqref{11} for $j=k$, we compute
\begin{align*}
m_kw_k+\dots+m_tw_t
&\le(m_k+\dots+m_t)w_k\\
&=\deg(\p_{\nu_k}g_k\cdots\p_{\nu_t}g_t)\\
&=p_k+\dots+p_t-w_{\nu_k}-\dots-w_{\nu_t}\\
&\le w_k+\dots+w_{t+1}-k-w_{\nu_k}-\dots-w_{\nu_t}.
\end{align*}
and hence
\[
(m_k-1)w_k+\dots+(m_t-1)w_t\le w_{t+1}-k-w_{\nu_k}-\dots-w_{\nu_t}.
\]
By \eqref{1}, this forces 
\begin{gather}
m_k=\dots=m_t=1,\label{35}\\
\nonumber w_{t+1}\ge w_{\nu_k}+\dots+w_{\nu_t}+k.
\end{gather}
In particular,
\begin{equation}\label{34}
\nu_k,\dots,\nu_t\ge t+2
\end{equation}
and hence $k\ge t-d+2$.
\end{proof}

%%%%%%%%%%%%%%%%%%%%%%%%%%%%%%%%%%%%%%%%%%%%%%%%%%%%%%%%%%%%%%%%%%%%%%%%%%%%%%%
\section{ICIS of embedding dimension $5$}\label{31}
%%%%%%%%%%%%%%%%%%%%%%%%%%%%%%%%%%%%%%%%%%%%%%%%%%%%%%%%%%%%%%%%%%%%%%%%%%%%%%%

\begin{lem}\label{33}
Let $A$ be a quasihomogeneous normal ICIS such that $\Theta_{A,<0}\ne0$.
Then $\fA(k_1)$ and $\fB(k_2)$ for $\{k_1,k_2\}=\{1,2\}$ is impossible.
\end{lem}

\begin{proof}
Assuming the contrary, one of the $g_j$ has a monomial divisible by $x_{k_1}^2$ by $\fA(k_1)$ and each of the $g_j$ has a monomial divisible by $x_{k_2}$ by $\fB(k_2)$.
In particular, 
\[
p_1+\dots+p_t\ge 2w_{k_1}+(t-1)w_{k_2}\ge w_1+\dots+w_{t+1}
\]
contradicting \eqref{5}.
\end{proof}

\begin{prp}\label{38}
For any quasihomogeneous ICIS $A$ as in \eqref{24} with $n=5$ and $t=2$, we have $\Theta_{A,<0}=0$.
\end{prp}

\begin{proof}
Assume that $\Theta_{A,<0}\ne0$.
By Proposition~\ref{14} and Lemma~\ref{33}, we must have $\fB(1)$ and $\fB(2)$.
Using \eqref{1}, \eqref{35}, and \eqref{34}, we may write
\begin{align*}
g_1&=x_1x_4+c_1x_2^jx_{k_1}+\cdots\\
g_2&=x_1x_5+c_2x_2x_{k_2}+\cdots
\end{align*}
with $\{k_1,k_2\}=\{4,5\}$ and $c_1,c_2\in K^*$.
As in the proof of Lemma~\ref{33}, the inequality~\eqref{5} can only hold true if $j=1$.
In this case, 
\[
A/(J_A+\ideal{x_3,\dots,x_n})=K\llangle x_1,x_2\rrangle/\ideal{\left\vert\frac{\p g}{\p(x_4,x_5)}\right\vert}.
\]
for degree reasons (see \eqref{1}), and hence $J_A$ is not $\mm_A$-primary.
This contradicts to the isolated singularity hypothesis.
\end{proof}

%%%%%%%%%%%%%%%%%%%%%%%%%%%%%%%%%%%%%%%%%%%%%%%%%%%%%%%%%%%%%%%%%%%%%%%%%%%%%%%
\section{Counter-examples}\label{18}
%%%%%%%%%%%%%%%%%%%%%%%%%%%%%%%%%%%%%%%%%%%%%%%%%%%%%%%%%%%%%%%%%%%%%%%%%%%%%%%

\begin{proof}[Proof of Example~\ref{22}]
The sequence $g$ is clearly regular and defines a complete intersection as in \eqref{24}.
Note that $\eta$ in \eqref{21} agrees with $\eta=\delta_{1,2,3}$ in \eqref{7}.
Since $\deg(g_1)=10=\deg(g_2)$, \eqref{4} shows that $\eta$ has negative degree $\deg\eta=-1$.

It remains to check that $A$ has an isolated singularity, that is, the Jacobian ideal $J_A$ from \eqref{25} is $\mm_A$-primary.
To this end, we may assume that $K=\bar K$ which enables us to argue geometrically on the variety 
\[
\bar X:=\Spec\bar A\subset\AA_K^n
\]
with $\bar A$ as in \eqref{36} using the Nullstellensatz.

The ideal $J_A$ is the image in $A$ of the Jacobian ideal $\bar J_g\unlhd\bar P$ of $g$ generated by the $2\times2$-minors 
\[
M_{i,j}:=\left\vert\frac{\p g}{\p(x_ix_j)}\right\vert
\]
of the Jacobian matrix of $g$ which reads
\[
\frac{\p g}{\p x}=
\begin{pmatrix}
x_4 & x_5 & 2x_3 & x_1-5x_4^4 & x_2 & 0 & 5x_7^4 & \cdots & 5x_n^4\\ 
x_5 & x_6 & 2x_3 & 0 & x_1 & x_2+5x_6^4 & 5c_7x_7^4 & \cdots & 5c_nx_n^4
\end{pmatrix}.
\]
With this notation we have to show that 
\[
\Sing\bar X=V(g,\bar J_g)=\{0\}.
\]
Due to those $2\times2$-minors of $\frac{\p g}{\p x}$ which involve only the columns $3,7,8,9,\dots,n$, only one of components $x_3,x_7,x_8,x_9,\dots,x_n$ of any $x\in\Sing\bar X$ can be non-zero.
We may therefore reduce to the case $n\leq 7$.

Because of the $3$rd column of $\frac{\p g}{\p x}$, we have $\bar J_g\cap K[x_1,\dots,x_6]\supseteq x_3 I$ where 
\[
I:=\ideal{x_4-x_5,x_5-x_6,x_1-x_2,x_1-5x_4^4,x_2+5x_6^4}.
\]
Note that $V(I)$ is the $x_3$-axis which is not contained in $V(g)$. 
It follows that $\Sing\bar X\cap V(x_7)$ is contained in the hyperplane $V(x_3)$. 
Similarly because of the $7$th column of $\frac{\p g}{\p x}$ and setting $c:=c_7$, we have $\bar J_g\cap K[x_1,\dots,\widehat{x_3},\dots,x_7]\supseteq x_7 I'$ where 
\[
I':=\ideal{cx_4-x_5,cx_5-x_6,cx_2-x_1,x_1-5x_4^4,x_2+5x_6^4}.
\]
Using $c^9+1\ne0$, we find that $V(I')$ is the $x_7$-axis and conclude $\Sing\bar X\cap V(x_3)\subset V(x_7)$ as before.
Summarizing the two cases, $\Sing\bar X$ is in fact contained in $V(x_3,x_7)$. 

Fix a point $(x_1,x_2,0,x_4,x_5,x_6,0)\in\Sing\bar X$.
Successively using the the equations
\begin{align*}
M_{1,2}&=x_4x_6-x_5^2=0,\\
M_{2,5}&=x_1x_5-x_2x_6=0,\\
g_2&=x_1x_5+x_2x_6+x_6^5=0,\\
M_{4,5}&=x_1(x_1-5x_4^4)=0,\\
M_{5,6}&=x_2(x_2+5x_6^4)=0,
\end{align*}
we derive
\[
x_4=0\Rightarrow x_5=0\Rightarrow x_2x_6=0\Rightarrow x_6=0\Rightarrow x_1=x_2=0.
\]
Similarly $x_6=0$ leaves no possibility except $x=0$ and $x_5=0$ reduces to one of these two cases by $M_{1,2}=0$.

Assume now that $x_4,x_5,x_6$ are all non zero. 
Then the minors $M_{1,5}$, $M_{2,4}$, $M_{2,5}$, $M_{2,6}$ give equations
\[
x_1x_4=x_2x_5,\quad x_1=5x_4^4,\quad x_1x_5=x_2x_6,\quad x_2=-5x_6^4.
\]
Substituting into $g$, we obtain
\[
g_1=2x_1x_4-x_4^5=9x_4^5,\quad g_2=2x_2x_6+x_6^5=-9x_6^5
\]
and hence $x_4=x_6=0$ contradicting our assumption.
\end{proof}

%%%%%%%%%%%%%%%%%%%%%%%%%%%%%%%%%%%%%%%%%%%%%%%%%%%%%%%%%%%%%%%%%%%%%%%%%%%%%%%
\bibliographystyle{amsalpha}
\bibliography{negder}
%%%%%%%%%%%%%%%%%%%%%%%%%%%%%%%%%%%%%%%%%%%%%%%%%%%%%%%%%%%%%%%%%%%%%%%%%%%%%%%
\end{document}